\documentclass{amsart}

\usepackage{hyperref, color}

\theoremstyle{plain}
\newtheorem{theorem}                {Theorem}      [section]
\newtheorem*{theorem*}                {Theorem \ref{thm:appl}}
\newtheorem{proposition}  [theorem]  {Proposition}
\newtheorem{corollary}    [theorem]  {Corollary}

\theoremstyle{definition}

\numberwithin{equation}{section}

\begin{document}

\title[Remarks on biharmonic hypersurfaces in space forms]
{Remarks on biharmonic hypersurfaces in space forms}

\author[Costa-Filho]{Wagner Oliveira Costa-Filho}
\address{Campus Arapiraca, Federal University of Alagoas, CEP 57309-005, 
Arapiraca,  Alagoas, Brazil}
\email{fcow@bol.com.br}

\date{\today}

\begin{abstract}  
We consider closed biharmonic hypersurfaces in the  Euclidean sphere and prove a rigidity result under a suitable condition on the scalar curvature. Moreover, we  establish an  integral formula involving the position vector for  biharmonic hypersurfaces in  space forms.
\end{abstract}

\maketitle

\section{Introduction}

The theory of biharmonic maps plays a fundamental role in many branches of Partial Differential Equations and Differential Geometry. The notion of biharmonic maps, as a natural generalization of harmonic maps, was introduced in 1964 by Eells and Sampson \cite{eells}, and the related concept of $k$-energy was considered in \cite{eells1}, by Eells and Lemaire.

Let $(M,g)$ and $(\bar{M},\bar{g})$ be Riemannian manifolds. A smooth map $f:M\to \bar{M}$ is said to be \textit{biharmonic} if it is a critical point of the \textit{bienergy functional} $$E_2(f)=\frac{1}{2}\int_M|\tau(f)|^2\, dM,$$
under compactly supported variations. Here, $\tau(f)$ is the tension field of $f$ given by $\tau(f)=tr(\nabla df)$. Notice that $\tau(f)=0$ means the map is harmonic. In the pioneer work \cite{jiang}, Jiang derived the variational formulas of the bienergy (see Urakawa \cite{urakawa}, chapter two). The Euler-Lagrange equation of this variational problem is $$\tau_2(f)=\Delta(\tau(f))-tr[\bar{R}(df,\tau(f))df]=0,$$
where $\Delta$ is the rough Laplacian and $\bar{R}$ denote the curvature tensor of $\bar{M}$. We use the sign conventions for the Laplacian operator and for the curvature tensor as in \cite{urakawa}. A submanifold $M$ of $\bar{M}$ is called biharmonic if the isometric immersion $x:M \to \bar{M}$ is a biharmonic map. In this case, the mean curvature vector $\textbf{H}$ satisfies the equation $$\Delta(\textbf{H})-tr[\bar{R}(dx,\textbf{H})dx]=0.$$

The study of biharmonic submanifolds in Euclidean spaces was initiated independently and from a different point of view by Chen in \cite{chen}, which proposed the following conjecture that remains open, but verified in many partial results: \textit{Every biharmonic submanifold of the Euclidean space must be minimal}. 

For more details about biharmonic maps and submanifolds we refer to Oniciuc \cite{oniciuc} and the recent book of Ou and Chen  \cite{ou1}.

One of the essential questions on biharmonic submanifolds is to comprehend their geometric aspects also in general ambient manifolds, notably space forms. Their classification is an active research field, and for a recent survey we indicate Fetcu and Oniciuc \cite{fetcu} as well as Ou \cite{ou2}.

For instance, Maeta \cite{maeta1} proposed the following global problem: \textit{Any complete biharmonic submanifold of a non-positively curved manifold is a minimal one}. On the other hand, Balmu\c{s}, Montaldo and Oniciuc in \cite{balmus} conjectured that \textit{biharmonic submanifolds in the unit Euclidean sphere $\mathbb{S}^{n+1}$ must have constant mean curvature}.

We point out that the main examples of non-minimal biharmonic hypersurfaces in $\mathbb{S}^{n+1}$ are $\mathbb{S}^n(1/\sqrt{2})$ and $\mathbb{S}^{n_1}(1/\sqrt{2})\times \mathbb{S}^{n_2}(1/\sqrt{2})$, $n_1+n_2=n, n_1\neq n_2$ (see Balmu\c{s}, Montaldo and Oniciuc \cite{balmus1}). Here, $ 1/\sqrt{2}$ denotes the radius of the sphere.

It was recently proved by Maeta and Ou in \cite{maeta2} that if a closed biharmonic hypersurface $M$ in $\mathbb{S}^{n+1}$ has constant scalar curvature, then $M$ is minimal or it has non-zero constant mean curvature, which gives a partial affirmative answer to the conjecture proposed by Balmu\c{s}, Montaldo and Oniciuc. See also the recent paper \cite{bibi} by Bibi, Loubeau and Oniciuc for another contribution that supports this conjecture.

Now let us consider an isometric immersion $x:M^n\to \bar{M}^{n+1}(c)$ of an orientable connected Riemannian manifold $M^n$ into the complete simply connected space form $\bar{M}^{n+1}(c)$ of constant curvature $c.$ Let $A$ denote the Weingarten operator of $M$ in $\bar{M}$ with respect to a unit normal vector field and $H=\frac{1}{n}trA$ the mean curvature function of $M.$ It follows from the Gauss equation that
\begin{equation}\label{scalar}
    n(n-1)(S-c)=n^2H^2-|A|^2,
\end{equation}
where $S$ denotes the normalized scalar curvature of $x$ and $|A|^2$ the squared norm of $A.$ We will denote by $\nabla$ the gradient operator of $M$. We say that $M$ is a closed manifold if it is compact and without boundary.

In this paper, generalizing the main result in \cite{maeta2}, we prove that

\begin{theorem}
Let $x:M^n \to \mathbb{S}^{n + 1}(c)$ be a closed biharmonic hypersurface. If $$\langle \nabla H^2, \nabla S \rangle \leq 0 ,$$ then $M$ has constant mean curvature.
\end{theorem}

We recall that a hypersurface in a space form is said to be \textit{linear Weingarten} if $S=aH+b$ for some constants $a,b\in \mathbb{R}.$ The linear Weingarten hypersurfaces can be regarded as a extension of hypersurfaces with constant mean curvature or constant scalar curvature and have been studied by many authors (see e.g.\cite{MR4117507}). So, from the above result, we obtain

\begin{corollary}
Let $x:M^n \to \mathbb{S}^{n + 1}(c)$ be a closed biharmonic linear Weingarten hypersurface with $S\leq b$, then $M$ has constant mean curvature.
\end{corollary}

Our next result present an integral formula for biharmonic hypersurfaces involving the distance function. 
First, we introduce some definitions about
the position vector in a complete simply connected space form
(see e.g. Vlachos \cite{vlachos} and references therein).
Let $d:\bar{M}^{n+1}(c)\to \mathbb{R}$ be the distance function 
relative to the basis point $x_0\in \bar{M}^{n+1}(c).$ The 
position vector of an arbitrary hypersurface $M$ in $\bar{M}^{n+1}(c)$, $x_0\notin M$, is given by 
$X=\psi_c(d)\bar{\nabla}d$, where $\bar{\nabla}d$ stands for the gradient of $d$ on $\bar{M}$ away from $x_0$, and $\psi_c(t)$ is the solution of 
the differential equation $y''+cy=0$ satisfying the boundary conditions $y(0)=0$ and $y'(0)=1.$ Denote by $\theta_c(t)=\frac{d}{dt}\psi_c(t)$ 
and $\theta_c = \theta_c(d)$.
In this setting we prove

\begin{theorem}
Let $x:M^n \to \bar{M}^{n + 1}(c)$ be a closed biharmonic hypersurface. Then,$$\int_M\theta_c(H^2-c)\,dM=0.$$
\end{theorem}

By using this result for $c>0$ we recover, directly and in a quite different method, Vieira's principal Theorem in \cite{vieira}. Moreover, if $c\leq 0$, we reobtain that 
closed biharmonic hypersurfaces in $\bar{M}^{n+1}(c)$ do not exist. 
In fact,  if $M$ lies in a hemisphere of $\mathbb{S}^{n+1}(c)$ then $\theta_c \geq 0$ and we always have that $\theta_c>0$ for $c\leq 0.$ 

\medskip

Finally, in the last section of this work, we presenting a new result 
regarding the Cheng-Yau operator on biharmonic hypersufaces
in spheres.

\section{Preliminaries}

This section contains some basic facts that we will use in order to prove our results. Consider the notations and conditions introduced above. Let $N$ be a unit normal vector field to $M.$ Bellow we denote by $\rho$ the support function of $x$ with respect to $x_0$, that is,  $\rho:M\to \mathbb{R}$ defined by $\rho(p) = \langle X(p), N(p) \rangle$ and 
$x^T$ the vector field on $M$ given by $x^T=X-\rho N.$  
In this context, we have 

\begin{proposition}[\cite{vlachos}, page 339] \label{1}
Let $ x: M^n \to \bar{M}^{n + 1}(c)$ be an orientable hypersurface. Then,
\begin{enumerate}
    \item[(a)] $\Delta(\rho)=-n\theta_cH-\rho|A|^2-n\langle \nabla H, x^T \rangle,$
    \item[(b)] $\nabla \rho=-A(x^T),$
\end{enumerate}
where $\Delta$ is the Laplace-Beltrami operator on $M$ such that $\Delta u=\emph{div}(\nabla u)$ for any smooth function  $u$ on $M$.
\end{proposition}

In the following we present the well known Minkowski integral formula. For the sake of completeness, we include a proof, which is  based on the approach of \cite{vlachos}.

\begin{proposition}\label{Minkowski}
Let $ x: M^n \to \bar{M}^{n + 1}(c) $ be a closed hypersurface. 
Then, $$ \int_M (\theta_c + H \rho) \, dM = 0.$$
\end{proposition}
\begin{proof}
Denote by $\nabla$ and $\bar{\nabla}$ the Riemannian connections of $M$ and $\bar{M}$, respectively. For any tangent vector $V$ to $M$ we have $\bar{\nabla}_VX=\theta_c(d)V$ (see \cite{vlachos}, page 338). Hence, by Weingarten formula we get $$\bar{\nabla}_Vx^T=\theta_c(d)V-V(\rho)N+\rho A(V).$$
Using $\bar{\nabla}_Vx^T=\nabla_Vx^T+\langle A(x^T),V\rangle N$ we obtain $$\nabla_Vx^T=\theta_cV+\rho A(V)-(V(\rho)+\langle A(x^T),V\rangle)N.$$
In particular, $$\nabla_Vx^T=\theta_cV+\rho A(V).$$
Let $\{e_1,...,e_n\}$ be an orthonormal local frame on $M$. Therefore, $$\textrm{div} (x^T)=\sum_{i=1}^n \langle \nabla_{e_i}x^T,e_i \rangle =n(\theta_c+H\rho ).$$
By integration we infer that $$\int_M (\theta_c + H \rho) \, dM = 0,$$ and the proof is concluded.
\end{proof}

 We recall the following fundamental characterization result for $M^n$ to be a biharmonic hypersurface in $\bar{M}^{n+1}(c)$, obtained by splitting the expression of the bitension field $\tau_2$ in its tangent and normal parts (see e.g. Ou \cite{ou3}, page 224, for a proof and  also Loubeau, Montaldo and Oniciuc \cite{loubeau} for case  of the submanifolds in an arbitrary ambient manifold).
 
 \begin{theorem}\label{2}
The isometric immersion $ x: M^n \to \bar{M}^{n + 1}(c)$ is biharmonic if and only if, satisfies the system of PDEs

$$
\left\{
\begin{array}{ll}
\Delta H=H|A|^2-ncH,\\
2A(\nabla H)=-nH\nabla H.
\end{array}  
\right.
$$
\end{theorem}

To conclude this section we present two identities that will be relevant
for our purposes. First, a straightforward calculation  yields 
\begin{equation}\label{laplaciano}
    \int_M(\Delta u)^2\,dM+\int_M\langle \nabla(\Delta u), \nabla u\rangle \,dM=0,
\end{equation}
where $M$ is a closed Riemannian manifold and $u$ a smooth function on $M$.

The next identity we need is the Bochner formula, which states that for any smooth function $u$ on $M$, we have
\begin{equation}\label{Bochner}
    \frac{1}{2}\Delta (|\nabla u|^2)= |\nabla ^2u|^2+\langle \nabla(\Delta u), \nabla u\rangle + \textrm{Ric}(\nabla u,\nabla u),
\end{equation}
where $\textrm{Ric}$ denotes the Ricci tensor of  $ M $ and $\nabla ^2u$ stands for the Hessian operator.

Making use of the inequality $|\nabla ^2u|^2\geq \frac{1}{n} (\Delta u)^2 ,$ it follows from (\ref{laplaciano}) and (\ref{Bochner}) that if $M$ is a closed Riemannian manifold then
\begin{equation}\label{Ricci}
    n\int_M\textrm{Ric}(\nabla u,\nabla u)\,dM+(n-1)\int_M\langle \nabla(\Delta u), \nabla u\rangle \,dM\leq 0.
\end{equation}

\section{Proof of Theorems}

In this final section we present the proofs of our results. For the reader's convenience, we restate the theorems.
Given an isometric immersion $x:M^n \to \bar{M}^{n + 1}(c)$, we recall that the Ricci tensor of $M$ is given by 
\begin{equation}\label{Ricci1}
    \textrm{Ric}(X,Y)=(n-1)c\langle X,Y \rangle+nH\langle A(X),Y \rangle -\langle A(X),A(Y)\rangle ,
\end{equation}
for any tangents vector fields $X,Y$ on $M.$

\begin{theorem}
Let $x:M^n \to \mathbb{S}^{n + 1}(c)$ be a closed biharmonic hypersurface. If $$\langle \nabla H^2, \nabla S \rangle \leq 0 ,$$ then $M$ has constant mean curvature.
\end{theorem}
\begin{proof}
From identity (\ref{scalar}) we have $$\nabla |A|^2=2n^2H\nabla H-n(n-1)\nabla S.$$

This implies that $$\langle \nabla(\Delta H), \nabla H\rangle =(|A|^2+2n^2H^2-nc)|\nabla H|^2-\frac{n(n-1)}{2}\langle \nabla H^2, \nabla S \rangle ,$$
taking into account the first equation in Theorem \ref{2}.

On the other hand, using (\ref{Ricci1}) and the second equation of Theorem \ref{2} to get $$\mathrm{Ric}(\nabla H,\nabla H)=((n-1)c-\frac{3}{4}n^2H^2)|\nabla H|^2.$$

Thus,  
\begin{eqnarray*}
n\mathrm{Ric}(\nabla H,\nabla H)+(n-1)\langle \nabla(\Delta H), \nabla H\rangle&=&(\frac{5}{4}n^3H^2-2n^2H^2+(n-1)|A|^2)|\nabla H|^2\\
&&-\frac{n(n-1)^2}{2}\langle \nabla H^2, \nabla S \rangle .
\end{eqnarray*}

From the well known inequality $|A|^2\geq nH^2$, with equality everywhere if and only if $M$ is umbilical, we obtain $$n\mathrm{Ric}(\nabla H,\nabla H)+(n-1)\langle \nabla(\Delta H), \nabla H\rangle \geq (\frac{5}{4}n^2-n-1)nH^2|\nabla H|^2-\frac{n(n-1)^2}{2}\langle \nabla H^2, \nabla S \rangle .$$

By inequality (\ref{Ricci}) and our hypothesis we conclude that \begin{eqnarray*}
0&\geq &(\frac{5}{4}n^2-n-1)n\int_MH^2|\nabla H|^2\,dM-\frac{n(n-1)^2}{2}\int_M\langle \nabla H^2, \nabla S \rangle \,dM \\
&\geq& (\frac{5}{4}n^2-n-1)n\int_MH^2|\nabla H|^2\,dM\geq 0.
\end{eqnarray*}
Hence, $H^2|\nabla H|^2=0$ on $M$, and this implies that
$M$ has constant mean curvature.
\end{proof}

Now we prove our second result:

\begin{theorem}
Let $x:M^n \to \bar{M}^{n + 1}(c)$ be a closed biharmonic hypersurface. Then,$$\int_M\theta_c(H^2-c)\,dM=0.$$
\end{theorem}
\begin{proof}
Multiplying by $H$ the first formula in Proposition \ref{1} and substituting the second identity of Theorem \ref{2} we get $$H\Delta(\rho)=-n\theta_cH^2-H\rho|A|^2+2\langle A(\nabla H), x^T \rangle .$$

Using the symmetry of operator $A$ and Proposition \ref{1} once more, it follows that $$H\Delta(\rho)=-n\theta_cH^2-H\rho|A|^2-2\langle \nabla H, \nabla \rho \rangle.$$

Integrating this identity and from the Divergence Theorem,  we get $$\int_MH\Delta (\rho)\,dM=-n\int_M\theta_cH^2\,dM-\int_MH\rho|A|^2\,dM+2\int_MH\Delta (\rho)\,dM.$$ 

Thus, $$-\int_M\rho \Delta (H)\,dM=-n\int_M\theta_cH^2dM-\int_MH\rho|A|^2\,dM.$$

By the first identity of Theorem \ref{2} we have, $$-\int_M\rho (H|A|^2-ncH) \,dM=-n\int_M\theta_cH^2\,dM-\int_MH\rho|A|^2\,dM.$$

So, it follows $$\int_McH \rho \, dM+\int_M\theta_cH^2\,dM=0.$$

Applying Minkowski formula of the Proposition \ref{Minkowski} we obtain $$\int_M\theta_c(H^2-c)\,dM=0.$$

This concludes the proof.
\end{proof}

Now let us consider the case of the unit sphere, that is 
$c=1$. If we choose $x_0 = e_{n+2}$ as the base point for the distance
function on $\mathbb S^{n+1}\subset \mathbb R^{n+2}$,
then we easily see that 
$\theta(d) = \cos d = \langle X, e_{n+2}\rangle$. 
Thus, setting $f=\langle X, e_{n+2}\rangle$ as in Vieira \cite{vieira}
we get 
$$\int_M f(H^2-1)\, dM= 0.$$ 
From this equality we have
\begin{theorem}[\cite{vieira}]
Let $M^n$ be a closed biharmonic hypersurface in a closed hemisphere of $\mathbb S^{n+1}$. 
If  $1-H^2$ does not change sign, then either 
$M^n$ is the equator of the hemisphere or it is the small hypersphere $\mathbb S^n(1/\sqrt 2)$.
\end{theorem}

\section{Biharmoninc hypersurfaces and the Cheng-Yau operator}
\label{l1}

We end this paper with the following observation. 
The Newton operator $P_1:TM\to TM$ associated to the hypersurface $x$ is defined by $P_1=nHI-A.$ 
Consider the Cheng-Yau  operator $L_1$ given by 
$$
L_1(u)=\textrm{div}(P_1(\nabla u)),
$$ 
for any smooth function $u$ on $M.$ The second order linear differential operator $L_1$ arises as the linearized operator of the scalar curvature for normal variations of the hypersurface. If $M$ is closed, then
\begin{equation}\label{div}
        \int_MuL_1 (v)\, dM = - \int_M \langle P_1 (\nabla u), \nabla v \rangle \, dM,
\end{equation}
whenever  $ u $ and $ v $ are smooth functions on  $ M .$ We refer the readers to \cite{rosenberg} for more details about this operator.

We are in position to state and prove the following

\begin{proposition}
Let $x:M^n \to \mathbb{S}^{n + 1}$ be a closed biharmonic hypersurface with $H\geq 0.$ If $L_1(H)\geq 0$,  then $M$ has constant mean curvature.
\end{proposition}
\begin{proof}
Initially, using the definition of $P_1$ and the second equation of Theorem \ref{2}, we note that $$\langle P_1 (\nabla H), \nabla H \rangle = \frac{3n}{2}H|\nabla H|^2.$$

By hypothesis and the identity (\ref{div}), we get $$0\geq -\frac{3n}{2} \int_M H|\nabla H|^2\,dM=-\int_M\langle P_1 (\nabla H), \nabla H \rangle \, dM=\int_MHL_1 (H)\, dM \geq 0.$$

Therefore, $H|\nabla H|^2=0$ on $M$  and
the result then follows directly.
\end{proof}


\subsection*{Acknowledgements} 
The author wishes to thank Professor Marcos P. Cavalcante for his interesting comments and helpful discussion during the preparation of this article. 
The author is also grateful to Professors Dorel Fetcu and Cezar Oniciuc for their interest and many valuables suggestions that 
have improved this article.

\bibliographystyle{amsplain}
\bibliography{bib}
\end{document}